\documentclass{article}
\usepackage[left=20truemm, right=20truemm]{geometry}
\usepackage{amsmath,amssymb,amsthm}
\usepackage{bm}
\usepackage{mathtools}
\usepackage[shortlabels]{enumitem}
\usepackage{float}
\usepackage{cite}
\usepackage{blkarray}
\usepackage{comment}
\usepackage{xspace}

\usepackage{xcolor,graphicx}
\usepackage{tikz}
\usetikzlibrary{calc}
\usepackage{subcaption}
\usepackage[pdfauthor={derajan},pdfstartview=XYZ,bookmarks=true,
colorlinks=true,linkcolor=blue,urlcolor=magenta,citecolor=blue
,bookmarks=true,linktocpage=true,hyperindex=true]{hyperref}

\theoremstyle{definition}
\newtheorem{theorem}{Theorem}[]

\newtheorem{definition}[theorem]{Definition}

\newtheorem{lemma}[theorem]{Lemma}
\newtheorem{claim}{Claim}[section]

\newcommand{\listcoloring}{{\sc List Coloring}\xspace}
\newcommand{\listpacking}{{\sc List Packing}\xspace}

\numberwithin{equation}{section}

\title{List packing of graphs with bounded tree-width}
\author{
Masaki Kashima\thanks{Keio University, Kanagawa, Japan. email: masaki.kashima10@gmail.com. This research is supported by JSPS KAKENHI Grant Number JP25KJ2077.}\qquad 
Shun-ichi Maezawa\thanks{Nihon University, Tokyo, Japan. email:maezawa.mw@gmail.com. This research is supported by JSPS KAKENHI Grant Number JP22K13956 and JSPS KAKENHI Grant Number JP25K17301.}\qquad
Xuding Zhu\thanks{Zhejiang Normal University, Jinhua, Zhejiang Province, China. email: xdzhu@zjnu.edu.cn. This research is supported by Natural Science Foundation of China, with grant number: NSFC 12371359.}
}

\begin{document}

\maketitle

\begin{abstract}
    Assume $L$ is a $k$-assignment of a graph $G$. An $L$-packing $\phi$ of $G$ is a sequence $\phi=(\phi_1, \ldots, \phi_k)$ of  $k$-mappings such that each $\phi_i$ is an $L$-coloring of $G$, and for each vertex $v$ of $G$, $\{\phi_1(v), \ldots, \phi_k(v)\} = L(v)$ (and hence $\phi_i(v) \ne \phi_j(v)$ when $i \ne j$).  We say $G$ is list $k$-packable if for any $k$-assignment $L$ of $G$, there is an $L$-packing of $G$. The list packing number $\chi_l^{\star}(G)$ of $G$ is the minimum
    integer $k$ such that $G$ is $k$-packable. For a positive integer $d$, let $t(d)$ be the maximum packing number of graphs of tree-width at most $d$.
    It was known that $d+1 \le t(d) \le 2d$  for any $d$. In this paper, we prove that $t(d) \le 2d-1$ for $d \ge 3$,  and $t(d) \ge d+2$ for $d \ge 2$. In particular, $t(2)=4$ and $t(3)=5$.  
    Furthermore, we show that for constant positive integers $k, d$, the problem of determining $\chi_l^{\star}(G)\leq k$ or not for a graph $G$ of tree-width at most $d$   is solvable in linear time.
\end{abstract}
\textbf{Keywords}: list coloring, list packing, tree-width, $d$-tree

\section{Introduction}\label{intro}

Throughout the paper, we only consider simple, finite, and undirected graphs.
For positive integers $a$ and $b$ with $a<b$, let $[a,b]$ denote the set of integers at least $a$ and at most $b$.

A list assignment of a graph $G$ is a mapping $L$ that assigns to each vertex $v$ of $G$ a set $L(v)$ of permissible colors. Given a list assignment $L$ of $G$, an {\em $L$-coloring} of $G$ is a mapping $\phi$ that assigns to each vertex $v$ of $G$ a color $\phi(v) \in L(v)$ so that $\phi(u) \ne \phi(v)$ for any edge $uv$ of $G$. A $k$-assignment of $G$ is a list assignment $L$ with $|L(v)| =k$ for each vertex $v$. We say $G$ is {\em $k$-choosable} if $G$ admits an $L$-coloring for every $k$-assignment $L$ of $G$. The {\em list chromatic number} (or {\em choice number}) $\chi_l(G)$ of $G$ is the minimum $k$ such that $G$ is $k$-choosable.

Given a $k$-assignment $L$ of $G$, an {\em $L$-packing} is a sequence $\phi=(\phi_1,\ldots, \phi_k)$ of $k$ $L$-colorings of $G$ such that for each vertex $v$, $\phi_i(v) \ne \phi_j(v)$ for any $1 \le i < j \le k$. As $\phi_i(v) \in L(v)$ for each $i$ and $|L(v)|=k$, it follows that 
$(\phi_1(v),\ldots, \phi_k(v))$ is a permutation of $L(v)$. We say $G$ is {\em list $k$-packable} if $G$ admits an $L$-packing for every $k$-assignment $L$ of $G$. The {\em list packing number} $\chi_l^{\star}(G)$ of $G$ is the minimum integer $k$ such that $G$ is  list $k$-packable.  The concept of list $k$-packing was introduced in~\cite{CCDK21}, and has since been further explored in~\cite{CCDK23,CCZ25+,CS2025,M23,KMMP22}.

A graph $G$ is {\em $d$-degenerate} if every subgraph $H$ of $G$ has minimum degree $\delta(H) \le d$. It is easy to see that $d$-degenerate graphs $G$ have $\chi_l(G) \le d+1$, and for any positive integer $d$, the complete graph $K_{d+1}$ is $d$-degenerate and have $\chi_l(K_{d+1})=d+1$.

It was proved in \cite{CCDK23} that if $G$ is $d$-degenerate,  then $\chi_l^{\star}(G) \le 2d$. On the other hand, it is known \cite{CCDK21} that for $d \ge 2$, there exists a $d$-degenerate graph $G$ with $\chi_l^{\star}(G) \ge d+2$.
Note that the constructed $d$-degenerate graph with $\chi_l^{\star}(G) \ge d+2$ has tree-width much larger than $d$.
The exact value of the maximum list packing number of $d$-degenerate graphs remains an open problem for $d \ge 3$. 

In this paper, we are interested in the list packing number of graphs of bounded tree-width, which is an important subclass of graphs with bounded degeneracy.

\begin{definition}
    \label{def-dtree}
    A graph $G$ is a $d$-tree if and only if $|V(G)| \ge d+1$, and its vertices can be ordered as $v_1, v_2, \ldots, v_n$ such that $\{v_1, v_2, \ldots, v_{d+1}\}$ induces a clique, and for $i \ge d+2$, $N_G(v_i) \cap \{v_1, v_2, \ldots, v_{i-1}\}$ induces a $d$-clique. The {\em tree-width} of a graph $G$ is the minimum $d$ such that $G$ is a subgraph of a $d$-tree.
\end{definition}

For a positive integer $d$, let 
$$t(d) = \max\{\chi_l^{\star}(G): G \text{ is a graph of tree-with $d$}\}.$$
As the list packing number $\chi_l^{\star}$ is closed under the subgraph relation, $t(d)$ is, in fact, equal to the least positive integer $k$ such that every $d$-tree is list $k$-packable.
It follows from the definition that each $d$-tree $G$ is $d$-degenerate, and hence $\chi_l^{\star}(G) \le 2d$. As $G$ contains a $(d+1)$-clique, we have $\chi_l^{\star}(G) \ge \chi_l(G) =d+1$. 
Thus, the known bound of $t(d)$ is $d+1 \le t(d) \le 2d$. 
 
In this paper, we improve each of the upper and lower bounds by $1$ as follows.

\begin{theorem}\label{thm:upper}
    For every integer $d\geq 3$ and every $d$-tree $G$, $\chi_l^{\star}(G)\leq 2d-1$.
\end{theorem}

\begin{theorem}\label{thm:lower}
For $d \geq 2$, there is a $d$-tree $G$ such that $\chi_l^{\star}(G) \geq d+2$.
\end{theorem}

As a result, we have the bound $t(d)\leq 2d-1$ for $d\geq 3$ and $t(d)\geq d+2$ for $d\geq 2$.
In particular, for $d=2,3$, we determine the exact value of $t(d)$ to be $t(2)=4$ and $t(3)=5$. 

In the last section, we investigate the complexity of the list packing number of graphs with bounded tree-width.
For the class of graphs with bounded tree-width, 
Fellows et al.~\cite{FFLRSST11} proved that the following \listcoloring problem, parameterized by the tree-width bound $t$, is fixed-parameter tractable and solvable in linear time for every fixed $t$.

\begin{description}
        \item[\underline{\listcoloring}]
    	\item[Input:] A graph $G$ of tree-width at most $t$, and a positive integer $k$.
    	\item[Output:] YES if $\chi_l(G) \leq k$, and NO otherwise.
\end{description}

In this paper, we prove that the following \listpacking problem, parameterized by the tree-width bound $t$, is also fixed-parameter tractable and solvable in linear time for every fixed $t$.


\begin{description}
        \item[\underline{\listpacking}]
    	\item[Input:] A graph $G$ of tree-width at most $t$, and a positive integer $k$.
    	\item[Output:] YES if $\chi^{\star}_l(G) \leq k$, and NO otherwise.
\end{description}


\begin{theorem}\label{thm:complexity}
The \listpacking problem, parameterized by the tree-width bound $t$, is fixed-parameter tractable and solvable in linear time for every fixed $t$.
\end{theorem}

The proof of Theorem~\ref{thm:complexity} uses the Courcelle's Theorem~\cite{Courcelle90}, which states that a statement written in MSO (monanic second order logic) can be determined in a linear time on graphs of a bounded tree-width.

In the rest of the paper, we give a proof of Theorems~\ref{thm:upper}, \ref{thm:lower}, and \ref{thm:complexity} in Sections~\ref{section:upper}, \ref{section:lower} and \ref{section:complexity}, respectively.









\section{Upper bound}\label{section:upper}

In our proof of Theorem~\ref{thm:upper}, we prove a little stronger statement which enables us to use an induction.
More precisely, we add the vertices one by one in the order given in Definition~\ref{def-dtree}.
At each step, assuming that $G\setminus \{v\}$ admits an $L$-packing, we try to extend   the colorings in the $L$-packing of $G\setminus \{v\}$ to $v$.
However, the extension may fail depending on the colorings of the neighbors of $v$.
Therefore, we proceed with the induction in such a way that such bad situations do not occur.
Assume $d \ge 2$ is an integer, $G$ is a $d$-tree, and $L$ is a $(2d-1)$-assignment of $G$.
Suppose that  $\phi=(\phi_1,\dots ,\phi_{2d-1})$ is an $L$-packing of $G$.
\begin{definition}
    A $d$-clique $K=\{u_1,u_2,\dots ,u_d\}$ of $G$ is   \textit{bad} (with respect to $\phi$) if 
\begin{center}
    there is a set $I\subseteq [1,2d-1]$ of size $d$ such that $\left|\bigcup_{i\in I}\phi_i(K)\right|=d$.
\end{center}
\end{definition}

In this section, we prove the following, which obviously implies Theorem~\ref{thm:upper}.

\begin{theorem}\label{thm:upper technical statement}
    Let $G$ be a $d$-tree for some integer $d\geq 3$ and let $L$ be a $(2d-1)$-assignment of $G$.
    Then  $G$ has an $L$-packing $\phi$ 
    without bad $d$-cliques. 
\end{theorem}

\subsection{Proof of Theorem~\ref{thm:upper technical statement}}
 
For convenience, we prove the result for a slightly larger family of graphs, namely, the family of {\em chordal graphs} with clique size at most $d+1$. Recall that a graph $G$ is chordal if and only if 
each induced subgraph has a {\em simplicial vertex}, i.e., a vertex $v$ such that $N_G(v)$ induces a clique in $G$. 

Assume $d \ge 3$ and $G$ is a chordal graph of clique size at most $d+1$. 
Let $L$ be a $(2d-1)$-assignment of $G$.
We shall prove that $G$ has an $L$-packing $\phi$ with no bad $d$-clique.

We prove by induction on the number of vertices of $G$.  If 
 $G$ consists of a single vertex $v$, then let $(\phi_1(v),  \ldots, \phi_{2d-1}(v))$ be any permutation of $L(v)$.
It is obvious that  $\phi=( \phi_1, \ldots, \phi_{2d-1})$ is an $L$-packing of $G$ with no bad $d$-clique.

Assume $ |V(G)| \ge 2$. Let $v$ be a simplicial vertex of $G$ and let  $U =N_G(v)$, which induces a clique of order at most $d$.  By induction hypothesis,
$G-v$ has an $L$-packing $\phi=( \phi_1, \ldots, \phi_{2d-1})$  with no bad $d$-clique. We shall extend each $\phi_i$ to an $L$-coloring of $G$ 
so that $\phi=( \phi_1, \ldots, \phi_{2d-1})$ is an $L$-packing of $G$ with no bad $d$-clique.

Set $X=\{\phi_1,\dots ,\phi_{2d-1}\}$ and $Y=L(v)$.
Let $H$ be the bipartite graph with parts $X$ and $Y$ such that $\phi_i\in X$ and $c\in Y$ are adjacent if and only if $c\notin \phi_i(U)$. To extend $\phi$ to an $L$-packing of $G$, we need to order the colors in $L(v)$ as $(c_1, c_2, \ldots, c_{2d-1})$ so that $ c_j$ is adjacent to $\phi_j$ in $H$ for $j=1,2, \ldots, 2d-1$.  By letting $\phi_i(v)=c_i$, we obtain an $L$-packing of $G$. In other words, we need to find a perfect matching $M$ in $H$. Indeed, an extension of $\phi$ to an $L$-packing of $G$ is equivalent to a perfect matching in $H$.

For any $\phi_i\in X$,  $|\phi_i(U)|=|U|\le d$. So $d_H(\phi_i)\geq (2d-1)-d=d-1$.
For each color $c\in Y$, for each $u \in U$, $c$ appears at most once in the set $\{\phi_i(u): 1\leq i\leq 2d-1\}$. So $c$ appears in at most 
$|U| \le d$ of the sets $\{ \phi_i(U): 1\leq i\leq 2d-1\}$. 
Thus $d_H(c)\geq (2d-1)-d=d-1$.

\begin{claim}\label{claim:matching existence}
    $H$ has a perfect matching.
\end{claim}

\begin{proof}
    Assume to the contrary that $H$ has no perfect matching.
    Then, by Hall's theorem, there is a set $S\subseteq X$ such that $|N_H(S)|<|S|$.
    As $\delta(H)\geq d-1$, we have $|S|\geq |N_H(S)|+1\geq (d-1)+1=d$.
    On the other hand, since $Y\setminus N_H(S)\neq \emptyset$, a vertex $c\in Y$ has at least $d-1$ neighbors in $X\setminus S$, forcing that $|S|\leq (2d-1)-(d-1)=d$.
    Combining these, we infer that $|S|=d$, $|N_H(S)|=d-1$, and $|Y\setminus N_H(S)|=d$.
    In particular, it must be that $|U|=d$.
    Without loss of generality, we may assume that $S=\{\phi_1,\dots ,\phi_d\}$ and $Y\setminus N_H(S)=[1,d]$.
    Then $\phi_i(U)=[1,d]$ for every $i\in [1,d]$, and hence $\left|\bigcup_{i=1}^d\phi_i(U)\right|=d$.
    This implies that $U$ is a bad $d$-clique of $G-v$ with respect to $\phi$, a contradiction.
\end{proof}

It follows from Claim \ref{claim:matching existence} that $\phi$ can be extended to an $L$-packing of $G$. 
To prove that the extension of $\phi$ is an $L$-packing of $G$ with no bad $d$-clique, the matching $M$ used for the extension needs to have some extra property.

For $X'\subseteq X$ and $Y'\subseteq Y$, let $E_H(X',Y')$ denote the set of edges between $X'$ and $Y'$.
 We say  \textit{$X'$ is complete to $Y'$} if $E_H(X',Y')=\{xy : x\in X', y\in Y'\}$ and \textit{$X'$ is anti-complete to $Y'$} if $E_H(X',Y')=\emptyset$. 
Let $M$ be a perfect matching in $H$.
We say that $X'\subseteq X$ is \textit{$d$-singular with respect to $M$} if $|X'|=d$ and $E_H(X', N_M(X'))\subseteq M$.
Analogously, we define $d$-singularity with respect to $M$ for $Y^{\prime} \subset Y$.
Note that if $X'\subseteq X$ is $d$-singular with respect to $M$, then $Y':=N_M(X')\subseteq Y$ is also $d$-singular with respect to $M$.

\begin{lemma}\label{lem:bigraph}
    Assume $d \ge 4$, and  $H$ is a bipartite graph  with parts $X$ and $Y$ such that $|X|=|Y|=2d-1$ and $\delta(H) \ge d-1$. 
   If $H$ has a perfect matching,  
    then there exists a perfect matching $M$ of $H$ such that no subset $X'\subseteq X$ is $d$-singular with respect to $M$.
\end{lemma}

\begin{lemma}\label{lem:52bigraph}
    Assume   $H$ is a bipartite graph  with parts $X$ and $Y$ such that $|X|=|Y|=5$ and $\delta(H) \ge 2$,
    and $H$ has a perfect matching $M$.
    If some $X'\subseteq X$ is $3$-singular with respect to $M$ and $|(N_H(x)\cup N_H(x'))\setminus N_M(X^{\prime})|=2$ for every distinct $x,x'\in X'$, then $H$ has another perfect matching $M'$ such that no subset of $X$ is $3$-singular.
\end{lemma}

The proofs of Lemmas \ref{lem:bigraph} and \ref{lem:52bigraph} are delayed to the next two subsections. 
Now we use these two lemmas to complete the proof of Theorem \ref{thm:upper technical statement}.

Assume first that $d \ge 4$.
By Claim~\ref{claim:matching existence} and Lemma~\ref{lem:bigraph}, $H$ has a perfect matching $M$ without $d$-singular subset $X'\subseteq X$ with respect to $M$.
Now we extend $\phi$ to $v$ according to $M$; i.e. $\phi_i(v)\in L(v)$ is a neighbor of $\phi_i\in X$ in $M$.
We show that $G'$ has no bad $d$-cliques with respect to $\phi$.

For the sake of contradiction, we assume that $G$ has a bad $d$-clique with respect to $\phi$.
As $G-v$ does not have a bad $d$-clique, the bad $d$-clique must contain $v$.
Without loss of generality, we may assume that $\{u_1,\dots ,u_{d-1}\}\cup \{v\}$ is a bad $d$-clique.
Then there is a set of $d$ colorings $X'$, say $X'=\{\phi_1,\dots ,\phi_d\}$, such that 
\[\left|\bigcup_{i=1}^d\{\phi_i(u_1),\dots ,\phi_i(u_{d-1}),\phi_i(v)\}\right|=d.
\]
Note that $N_M(X')=\{\phi_1(v),\dots ,\phi_d(v)\}$.
As $N_M(X')$ is a set of $d$ colors contained in
$\bigcup_{i=1}^d\{\phi_i(u_1),\dots ,\phi_i(u_{d-1}),\phi_i(v)\}$, we infer that $\{\phi_i(u_1),\dots ,\phi_i(u_{d-1}), \phi_i(v)\}=N_M(X')$ for any $i\in [1,d]$.
This implies that the coloring $\phi_i\in X'$ has no neighbor in $N_M(X')\setminus \{\phi_i(v)\}$ in $H$ for each $i\in [1,d]$.
Thus, $X'\subseteq X$ is $d$-singular with respect to $M$, a contradiction.
Therefore, $G'$ has no bad $d$-clique with respect to $\phi$.

Next assume that $d=3$. 
By Claim~\ref{claim:matching existence}, $H$ has a perfect matching $M$.
We extend $\phi$ to an $L$-packing of $G$ according to $M$.
If there is no bad $d$-clique in $G$ with respect to $\phi$, then we are done.
Thus, we assume that $G$ has a bad $d$-clique with respect to $\phi$.
As the bad $d$-clique must contain $v$, without loss of generality, we may assume that $\{v,u_1,u_2\}$ is a bad $d$-clique due to the three colorings $\phi_1$, $\phi_2$, and $\phi_3$.
Furthermore, we may assume that $L(v)=\{1,2,3,4,5\}$ and $\phi_i(v)=i$ for $i\in \{1,2,3\}$.
By a similar argument as in the $d\geq 4$, $X':=\{\phi_1, \phi_2, \phi_3\}\subseteq X$ is $3$-singular with respect to $M$.
Thus $\{\phi_i(u_1), \phi_i(u_2), \phi_i(v)\} = \{1,2,3\}$ for $i \in \{1,2,3\}$. Hence
\begin{itemize}
    \item $\{\phi_1(u_1),\phi_1(u_2)\}=\{1,2,3\}\setminus \{\phi_1(v)\}=\{2,3\}$,
    \item $\{\phi_2(u_1),\phi_2(u_2)\}=\{1,2,3\}\setminus \{\phi_2(v)\}=\{1,3\}$, and 
    \item $\{\phi_3(u_1),\phi_3(u_2)\}=\{1,2,3\}\setminus \{\phi_3(v)\}=\{1,2\}$.
\end{itemize}
By definition, for $i\in \{1,2,3\}$, $\{4,5\} \setminus \{\phi_i(u_3)\} \subseteq L(v)-\phi_i(\{u_1,u_2,u_3\}) =  N_H(\phi_i)$.  Since $\phi_1(u_3) \ne \phi_2(u_3)$, we conclude that $(N_H(\phi_1)\cup N_H(\phi_2))\setminus \{1,2,3\}=\{4,5\}$.
Similarly, we have that $(N_H(\phi_2)\cup N_H(\phi_3))\setminus \{1,2,3\}=(N_H(\phi_3)\cup N_H(\phi_1))\setminus \{1,2,3\}=\{4,5\}$.
Then, by Lemma~\ref{lem:52bigraph}, $H$ has another perfect matching $M'$ such that no subset $X'\subseteq X$ is $3$-singular.
We unpack the vertices $v$ and re-extend $\phi$ to $v$ according to $M'$.
Then $\phi$ is an $L$-packing of $G$ without bad $d$-cliques, as desired.
This completes the proof of Theorem~\ref{thm:upper technical statement}, subject to verifications of Lemmas 
 \ref{lem:bigraph} and \ref{lem:52bigraph}. 
 
Now we prove these two lemmas in two subsections. 

\subsection{ Proof of Lemma \ref{lem:bigraph}}
Assume $d\geq 4$, and $H$ is a bipartite graph 
satisfying the condition of Lemma~\ref{lem:bigraph}.
Set $X=\{x_1,\dots ,x_{2d-1}\}$ and $Y=\{y_1,\dots ,y_{2d-1}\}$.

We first show that if there are two distinct $d$-singular subsets of $X$ with respect to some perfect matching, then we can find another matching without $d$-singular sets.

\begin{claim}\label{claim:singular unique}
    Let $M$ be a perfect matching of $H$.
    If $X'\subseteq X$ and $X''\subseteq X$ $(X'\neq X'')$ are both $d$-singular, then $H$ has another perfect matching $M'$ such that no subset of $X$ is $d$-singular with respect to $M'$.
\end{claim}

\begin{proof}
    Suppose that there are two distinct $d$-singular sets $X', X''\subseteq X$.
    Then $|X \cup X'| \geq d+1$.
    We set $Y'=N_M(X')$ and $Y''=N_M(X'')$.
    Without loss of generality, we may assume that $M=\{x_iy_i : 1\leq i\leq 2d-1\}$ and $X'=\{x_1,\dots ,x_d\}$.
    As $|X|=2d-1<|X'|+|X''|$, there is a vertex, say $x_d$, in $X'\cap X''$.
    Since $x_d$ has no neighbor in $(Y'\cup Y'')\setminus \{y_d\}$, we have that 
    \[
    |X'\cup X''|=|Y'\cup Y''|\leq (2d-1)-(d_H(x_d)-1)\leq (2d-1)-(\delta(H)-1)\leq d+1,
    \]
    and this together with $|X \cup X'| \geq d+1$ implies that $|X'\cup X''|=|Y'\cup Y''|=d+1$.
    Without loss of generality, we may assume that $X''=\{x_2,\dots ,x_{d+1}\}$.
    Then, for each $i\in [2,d]$, $y_i$ has no neighbors in $\{x_1,\dots ,x_{d+1}\}\setminus \{x_i\}$ and $x_i$ has no neighbors in $\{y_1,\dots ,y_{d+1}\}\setminus \{y_i\}$.
    This, together with the fact that $d_H(x_i)\geq \delta(H)\geq d-1$, implies that $\{x_2,\dots ,x_d\}$ is complete to $\{y_{d+2},\dots ,y_{2d-1}\}$ and $\{y_2,\dots ,y_d\}$ is complete to $\{x_{d+2},\dots ,x_{2d-1}\}$.
    
    Now we define another perfect matching $M^*$ of $H$ by
    \[M^*=\{x_1y_1,x_2y_2,x_{d+1}y_{d+1}\}\cup \{x_iy_{i+d-1} : 3\leq i\leq d\}\cup \{x_iy_{i-d+1} : d+2\leq i\leq 2d-1\}.
    \]
    Suppose that $X^*\subseteq X$ is $d$-singular with respect to $M^*$ and set $Y^*=N_M(X^*)$.
    If there are two distinct indices $i,j\in [2,d]$ such that $x_i,x_j\in X^*$, then since $\{x_2,\dots ,x_d\}$ is complete to $\{y_{d+2},\dots ,y_{2d-1}\}$ and
    $(N_{M^*}(x_i) \cup N_{M^*}(x_j))\cap \{y_{d+2},\dots ,y_{2d-1}\}\neq \emptyset$, either $x_i$ or $x_j$ has two neighbors in $Y^*$.
    Thus, we have that $|X^*\cap \{x_2,\dots ,x_d\}|\leq 1$.
    By exchanging the roles of $X$ and $Y$ in the above argument, we have $|Y^*\cap \{y_2,\dots ,y_d\}|\leq 1$ and this implies that $|X^*\cap (\{x_2\}\cup \{x_{d+2},\dots ,x_{2d-1}\})|\leq 1$.
    Combining these and the assumption that $d\geq 4$, it follows that $d=4$ and $X^*=\{x_1,x_{d+1},x_i,x_j\}$ for some $i\in [2,d]$ and $j\in [d+2,2d-1]$.
    Note that $Y^*=\{y_1,y_{d+1},y_{i+d-1},y_{j-d+1}\}$.
    Since $x_1\in X'$ and $X'$ is $d$-singular with respect to $M$, $x_1$ has at least $d_H(x_1)-|Y'\setminus \{y_1\}|\geq d-2$ neighbors in $\{y_{d+1},\dots ,y_{2d-1}\}$.
    Thus, $x_1$ is adjacent to at least one of $\{y_{d+1},y_{i+d-1}\}\subseteq Y^*$, implying that $X^*$ is not $d$-singular with respect to $M^*$, a contradiction.
    Thus, $M^*$ is a desired perfect matching of $H$.
\end{proof}

Suppose that $H$ has a perfect matching $M_0$.
If no subset of $X$ is $d$-singular with respect to $M_0$, then we are done.
Thus, we assume that $X_0\subseteq X$ is $d$-singular with respect to $M_0$, and set $Y_0=N_{M_0}(X_0)$.
By Claim~\ref{claim:singular unique}, we may assume that no other subset of $X$ is $d$-singular pair with respect to $M_0$.
Without loss of generality, we may assume that $M_0=\{x_iy_i : 1\leq i\leq 2d-1\}$, $X_0=\{x_1,\dots ,x_d\}$, and $Y_0=\{y_1,\dots ,y_d\}$. 

\begin{claim}\label{claim:alternating c4}
    For every edge $xy\in E_H(X_0,Y_0)$, $H$ has a 4-cycle $xy'x'yx$ such that $x'\in X\setminus X_0$, $y\in Y\setminus Y_0$, and $x'y'\in M_0$.
\end{claim}

\begin{proof}
    Let $xy$ be an edge in $E_H(X_0,Y_0)$.
    Since $X_0$ is $d$-singular and $\delta(H)\geq d-1$, $x$ has at least $d-2$ neighbors in $Y\setminus Y_0$.
    Similarly, since $y$ has no neighbors in $X_0\setminus \{x\}$, $y$ has at least $d-2$ neighbors in $X\setminus X_0$.
    As $|X\setminus X_0|=|Y\setminus Y_0|=d-1$ and $d\geq 4$, by the Pigeon-hole principle, there is $x'\in X\setminus X_0$ and $y'\in Y\setminus Y_0$ such that $x'y'\in M_0$, $x'y\in E(H)$, and $y'x\in E(H)$.
\end{proof}

We consider the edge $x_1y_1\in E_H(X_0,Y_0)$.
By Claim~\ref{claim:alternating c4}, we may assume that $H$ has a 4-cycle $x_1y_{d+1}x_{d+1}y_1x_1$.
Let $M_1:=(M_0\setminus \{x_1y_1,x_{d+1}y_{d+1}\})\cup \{x_1y_{d+1},x_{d+1}y_1\}$.
It is easy to verify that $M_1$ is again a perfect matching of $H$.
If $H$ has no $d$-singular pair with respect to $M_1$, then we are done.
Suppose that $X_1\subseteq X$ is $d$-singular with respect to $M_1$, and let $Y_1=N_{M_1}(X_1)$.

\begin{claim}\label{claim:m1 structure}
    Either  $X_1=X_0$ or $Y_1=Y_0$.
\end{claim}

\begin{proof}
    Assume $X_1\neq X_0$.
    If $\{x_1,x_{d+1}\}\cap X_1=\emptyset$, then we have $E_H(X_1,Y_1)\cap M_1=E_H(X_1,Y_1)\cap M_0$, and hence $X_1$ is $d$-singular with respect to $M_0$ as well, a contradiction to the uniqueness of $X_0$.
    Thus, we have either $x_1\in X_1$ or $x_{d+1}\in X_1$.
    Furthermore, since $y_1\in N_H(x_1)\cap N_H(x_{d+1})$, $X_1$ contains exactly one of $x_1$ and $x_{d+1}$, and hence $X_1 \cap \{x_2,\dots ,x_d\} \neq \emptyset$.
    Fix a vertex $x_i\in X_1 \cap \{x_2,\dots ,x_d\}$.
    Then $y_i$ has no neighbors in $(X_0\cup X_1)\setminus \{x_i\}$, and hence
    \[
    |X_0\cup X_1|\leq (2d-1)-(d_H(y_i)-1)\leq (2d-1)-(d-2)=d+1.
    \]
    Since $X_0\neq X_1$, we have $|X_0\cup X_1|=d+1$.
    Similarly, since $x_i$ has no neighbors in $(Y_0\cup Y_1)\setminus \{y_i\}$, we have that $|Y_0\cup Y_1|\leq d+1$.
    Suppose $x_1\in X_1$. Then there must be $x_j\in X_1\setminus \{x_1,\dots ,x_{d+1}\}$ and both $y_{d+1}=N_{M_1}(x_1)$ and $y_j=N_{M_1}(x_j)$ belong to $Y_1\setminus Y_0$.
    Hence we obtain $|Y_0 \cup Y_1| \geq |Y_0|+|\{y_{d+1},y_j\}| = d + 2$, a contradiction.
    Thus, it follows that $x_{d+1}\in X_1$ and $x_1\notin X_1$.
    Since $|X_0\cup X_1|=d+1$, we conclude that $X_1=\{x_2,\dots ,x_{d+1}\}$, implying that $Y_1=Y_0$.
\end{proof}

By symmetry, we may assume that $Y_1=Y_0$, and hence $X_1=N_{M_1}(Y_1)=\{x_2,\dots ,x_{d+1}\}$. 
Since every $y_i\in Y_0\setminus \{y_1\}$ has no neighbors in $\{x_1,\dots ,x_{d+1}\}\setminus \{x_i\}$ and $d_H(y_i)\geq d-1$, it follows that 
\begin{quote}
    (a) $Y_0\setminus \{y_1\}$ is complete to $\{x_{d+2},\dots ,x_{2d-1}\}$.
\end{quote}

Next, we consider the edge $x_2y_2\in M_0$.
By Claim~\ref{claim:alternating c4}, $H$ has a 4-cycle $x_2y_ix_iy_2x_2$ for some $x_i\in X\setminus X_0=\{x_{d+1},\dots ,x_{2d-1}\}$.
Since $y_2$ is not adjacent to $x_{d+1}$, without loss of generality, we may assume that $x_2y_{d+2}x_{d+2}y_2x_2$ is a 4-cycle of $H$.
Let $M_2:=(M_0\setminus \{x_2y_2,x_{d+2}y_{d+2}\})\cup \{x_2y_{d+2},x_{d+2}y_2\}$.
Then $M_2$ is a perfect matching of $H$.
If $H$ has no $d$-singular pairs with respect to $M_2$, then we are done.
Thus, we may assume that $X_2\subseteq X$ is $d$-singular, and set $Y_2=N_{M_2}(X_2)$.
By a similar argument to Claim~\ref{claim:m1 structure}, we obtain the following claim on $(X_2,Y_2)$.

\begin{claim}\label{claim:m2 structure}
    Either $X_2=X_0$ or $Y_2= Y_0$. 
\end{claim}

Furthermore, by (a), $x_{d+2}$ has at least $d-1\geq 3$ neighbors in $Y_0$, which implies that $X_2=X_0$.
Then every $x_i\in X_0\setminus \{x_2\}$ has no neighbors in $(\{y_1,\dots ,y_d\}\cup\{y_{d+2}\})\setminus \{y_i\}$ and $d_H(x_i)\geq d-1$, which implies that
\begin{quote}
    (b) $X_0\setminus \{x_2\}$ is complete to $\{y_{d+1}\}\cup\{y_{d+3},\dots y_{2d-1}\}$.
\end{quote}

Finally, we consider the third edge $x_3y_3\in M_0$.
The assumption $d\geq 4$ implies that $d+3\leq 2d-1$, and thus (a) and (b) imply that $x_3y_{d+3},x_{d+3}y_3\in E(H)$.
Let $M_3:=(M_0\setminus \{x_3y_3,x_{d+3}y_{d+3}\})\cup \{x_3y_{d+3},x_{d+3}y_3\}$.
Then $M_3$ is a perfect matching of $H$.
We shall show that no subset of $X$ is $d$-singular with respect to $M_3$.

For the sake of contradiction, suppose that $X_3\subseteq X$ is $d$-singular with respect to $M_3$, and set $Y_3=N_{M_3}(X_3)$.
Then, again by an argument similar to the Claim~\ref{claim:m1 structure}, 
either $X_3=X_0$ or $Y_3=Y_0$.
However, since $x_{d+3}$ is adjacent to at least $d-1\geq 3$ vertices in $Y_0$ by (a) and $y_{d+3}$ is adjacent to at least $d-1\geq 3$ vertices in $X_0$ by (b), neither  $X_0$ nor $X_0\setminus \{x_3\}\cup \{x_{d+3}\}=N_{M_3}(Y_0)$ is $d$-singular with respect to $M_3$, a contradiction.
This completes the proof of Lemma \ref{lem:bigraph}.

\subsection{Proof of Lemma \ref{lem:52bigraph}} 

Assume $H$ is a bipartite graph with partite set $X=\{x_1,x_2,x_3,x_4,x_5\}$ and $Y=\{y_1,y_2,y_3,y_4,y_5\}$, and satisfies the condition of Lemma \ref{lem:52bigraph}.

Without loss of generality, we may assume that $M=\{x_iy_i : 1\leq i\leq 5\}$ is a perfect matching in $H$.
Suppose further that $\{x_1,x_2,x_3\}\subseteq X$ is $3$-singular with respect to $M$, 
and $|(N_H(x_i)\cup N_H(x_j))\setminus \{y_1,y_2,y_3\}|=2$ for every distinct $i,j\in \{1,2,3\}$.
Obviously, the assumption implies that
\[
(N_H(x_1)\cup N_H(x_2))\setminus \{y_1,y_2,y_3\}=(N_H(x_2)\cup N_H(x_3))\setminus \{y_1,y_2,y_3\}=(N_H(x_3)\cup N_H(x_1))\setminus \{y_1,y_2,y_3\}=\{y_4,y_5\}.
\]
Without loss of generality, we may assume that $\{y_4,y_5\}\subseteq N_H(x_1)$, $y_4\in N_H(x_2)$, and $y_5\in N_H(x_3)$.
Since $d_H(y_i)\geq 2$ for each $i\in \{1,2,3\}$, each of $y_1,y_2,y_3$ is adjacent to a vertex in $\{x_4,x_5\}$.
By the symmetry of $(x_2,y_4)$ and $(x_3,y_5)$, we may assume that $y_1$ is adjacent to $x_4$.
We consider the following four cases.
\begin{itemize}
    \item If $x_4y_2, x_4y_3\in E(H)$, then let $M'=\{x_1y_1,x_2y_4,x_3y_3,x_4y_2,x_5y_5\}$.
    \item If $x_4y_2,x_5y_3\in E(H)$, then let $M'=\{x_1y_1,x_2y_4,x_3y_3,x_4y_2,x_5y_5\}$.
    \item If $x_4y_3,x_5y_2\in E(H)$, then let $M'=\{x_1y_4,x_2y_2,x_3y_3,x_4y_1,x_5y_5\}$.
    \item If $x_5y_2,x_5y_3\in E(H)$, then let $M'=\{x_1y_1,x_2y_2,x_3y_5,x_4y_4,x_5y_3\}$.
\end{itemize}
For each case, it is easy to verify that $X$ has no $3$-singular set with respect to $M'$.
This completes the proof of Lemma~\ref{lem:52bigraph}.

\section{Lower bound}\label{section:lower}

We prove Theorem~\ref{thm:lower}.
Let $H$ be a graph isomorphic to $K_{d+1}$ with $V(H) = \{v_1,v_2,\ldots,v_{d+1}\}$.
Let $G$ be a graph obtained from $H$ and  $d$ new vertices $w_1, w_2, \ldots ,w_d$ by connecting $w_i$ and $\{v_1,v_2,\ldots,v_d\}$ for $i \in [1,d+1]$.
We define a $(d+1)$-assignment $L$ of $G$ by
\[ L(v_i) = L(w_i) = [1,d+2] \setminus \{i\}\quad \forall i\in [1,d]\]
and $L(v_{d+1}) = [1,d+2]\setminus \{d+1\}$.

Assume $(\phi_1,\phi_2,\ldots,\phi_{d+1})$ is an $L$-list packing  of $G$.
For a subset $A$ of $[1,d+2]$, let $A^c = [1,d+2] \setminus A$ be the complement of $A$.
Assume $X$ is a $d$-clique of $G$. For each $i \in [1,d+1]$, $\phi_i(X)$ is a $d$-subset of $[1,d+2]$, and hence $\phi_i(X)^c$ is a 2-subset of $[1,d+2]$. Let $H_X$ be the multi-graph with  vertex set $[1,d+2]$ and  edge set 
$\{ \phi_i(X)^c: 1\leq i\leq d+1\}$.

\begin{lemma}\label{bad_clique}
There exists
a $d$-clique $X$ of $G$ such that $H_X$ contains a cycle.
\end{lemma}
\begin{proof}
 Assume to the contrary that for each $d$-clique $X$, $H_X$ is a forest (and hence a tree, as $H_X$ has $d+2$ vertices and $d+1$ edges). 
In particular, $H_X$ is a simple graph (a pair of parallel edge is a 2-cycle).  

Let $M = [m_{i,j}]$ be the $(d+1) \times (d+1)$ matrix with $m_{i,j} = \phi_i(v_j)$.
\[
M=
\begin{blockarray}{c@{\hspace{5pt}}cccc}
     & L(v_1) & L(v_2) & \cdots & L(v_{d+1}) \\
\begin{block}{c(cccc)}
  \phi_1 & \phi_1(v_1) & \phi_1(v_2) & \cdots & \phi_1(v_{d+1}) \\
  \phi_2 & \phi_2(v_1) & \phi_2(v_2) & \cdots & \phi_2(v_{d+1}) \\
  \vdots & \vdots & \vdots & \ddots & \vdots \\
  \phi_{d+1} & \phi_{d+1}(v_1) & \phi_{d+1}(v_2) & \cdots & \phi_{d+1}(v_{d+1})\\
\end{block}
\end{blockarray}
\]
For $k \in [1,d+1]$, let $X_k$ be the $d$-clique of $G$ induced by $  \{v_i: i \in [1, d+1] \setminus \{k\}\}$. 
Let $M_k$ be the matrix obtained from $M$ by deleting the $k$th column.  For example, $M_{d+1}$ is the submatrix of $M'$ consisting of the columns to the left of the vertical line. 

For any color $a \in [1,d+2] \setminus \{k, d+2\}$, there are $d-1$ columns of $M_k$ that contain $a$, and hence there are $d-1$ rows of $M_k$ that contain $a$ (note that each row contains at most one copy of $a$). This means that 
for $d-1$ indices $i$, $\phi_i(X_k)$ contains color $a$. Thus there are two indices $i$ for which 
$a \in \phi_i(X_k)^c$, i.e., $a$ has degree $2$ in $H_{X_k}$. 
For $a \in \{k, d+2\}$, all the $d$ columns 
contain the color $a$. Hence there are $d$ rows in $M_k$ containing $a$, and hence $a$ has degree $1$ in $H_{X_k}$.
Thus $H_{X_k}$ is a path with end vertices $k$ and $d+2$.
 
Assume $H_{X_{d+1}} = a_1a_2\ldots a_{d+2}$, which is a path connecting $a_1 = d+2$ and $a_{d+2} = d+1$.
By renaming the colorings, if needed, we may assume that $ \phi_i (X_{d+1})^c = \{a_i, a_{i+1}\}$. 

\begin{claim}
    \label{clm-1}
    For $ i \in [1,d+1]$, $\phi_i(v_{d+1})=a_i$ and $\phi_i(V(H))^c =\{a_{i+1}\}$.
\end{claim}
\begin{proof}
    Since $a_{d+2}=d+1 \notin L(v_{d+1})$ and $ \phi_{d+1}(v_{d+1}) \in  \phi_{d+1} (X_{d+1})^c = \{a_{d+1}, a_{d+2}\}$, we conclude that $\phi_{d+1}(v_{d+1})= a_{d+1}$ and $\phi_{d+1}(V(H))^c =\{a_{d+2}\}$.
    
Assume $i \le d$, and $\phi_{i+1}(v_{d+1}) = a_{i+1}$.  Since 
$\phi_{i}(v_{d+1})\in \phi_{i} (X_{d+1})^c = \{a_{i}, a_{i+1}\}$ and $\phi_{i}(v_{d+1}) \ne \phi_{i+1}(v_{d+1}) = a_{i+1}$, we conclude that  $\phi_i(v_{d+1})=a_i$ and $\phi_i(V(H))^c =\{a_{i+1}\}$.
\end{proof}
The matrix $M'$ illustrates Claim \ref{clm-1}.
The pairs to the right of the vertical line are the edges of $H_{X_{d+1}}$.

\[
M'=
\begin{blockarray}{c@{\hspace{5pt}}cccccc}
     & L(v_1) & L(v_2) & \cdots & L(v_d) & L(v_{d+1}) &  \\
\begin{block}{c(cccc|cc)}
  \phi_1 & \phi_1(v_1) & \phi_1(v_2) & \cdots & \phi_1(v_d) & a_1 & a_2 \\
  \phi_2 & \phi_2(v_1) & \phi_2(v_2) & \cdots & \phi_2(v_d) & a_2 & a_3 \\
  \vdots & \vdots & \vdots & \ddots & \vdots  & \vdots & \vdots \\
  \phi_{d+1} & \phi_{d+1}(v_1) & \phi_{d+1}(v_2) & \cdots & \phi_{d+1}(v_d) & a_{d+1} & a_{d+2} \\
\end{block}
\end{blockarray}
\]

\begin{claim}
    \label{clm-2}
    For $i \in [1,d]$, $\phi_i(w_{a_{d+1}})=a_i$ and $\phi_{d+1}(w_{a_{d+1}})=a_{d+2}$.
\end{claim}
\begin{proof}
    As $w_{a_{d+1}}$ is adjacent to every vertex in $X_{d+1}$, for each $i \in [1,d+1]$, $\phi_i(w_{a_{d+1}}) \in \phi_i(X_{d+1})^c =\{a_i, a_{i+1}\}$. As 
    \[
    \phi_d(X_{d+1})^c = \{a_d,a_{d+1}\},\quad \phi_{d+1}(X_{d+1})^c = \{a_{d+1},a_{d+2}\},\quad \text{and}\quad a_{d+1} \notin L(w_{a_{d+1}}),
    \]
    we conclude that $\phi_d(w_{a_{d+1}}) = a_d$ and $\phi_{d+1}(w_{a_{d+1}}) = a_{d+2}$. 
    Similarly to Claim \ref{clm-1}, we can prove by induction on $i$ that $\phi_i(w_{a_{d+1}}) = a_i$ for all $i\in [1,d]$. 
\end{proof}

Let $i^{\prime} \in [1,d]$ be the index such that $\phi_d(v_{i^{\prime}}) = d+2$.
Then  $H_{X_{i^{\prime}}} = b_1b_2\ldots b_{d+2}$ is a path connecting $b_1 = d+2$ and $b_{d+2}=i'$. 
Let $i_j \in [1,d+1]$ be the integer such that
$[1,d+2] \setminus \phi_{i_j}(V(H)) = \{b_{j+1}\}$ for $1 \leq j\leq d+1$.
Note that $M_{i'}$ is obtained from $M_{d+1}$ by replacing the column indexed by $L(v_{i'})$ with the column indexed by $L(v_{d+1})$, which consists of the columns to the right of the vertical line in the matrix below
(the rows in the augmented matrix below are re-arranged to the order $i_1i_2\ldots i_{d+1}$).
Since $\phi_d(v_{i^{\prime}}) = d+2 = b_1$, we have $i_1 = d$.
Moreover, since $[1,d+2] \setminus \phi_d(V(H)) = \{a_{d+1}\}$, we have $b_2 = a_{d+1}$.
Let $\ell \in [1,d+1]$ be the integer such that $i_{\ell} = d+1$.
Since $[1,d+2] \setminus \phi_{d+1}(V(H)) = \{a_{d+2}\} = \{d+1\}$, we have $b_{\ell+1} = d+1$ (see the following matrix $M''$, where the pairs to the right of the vertical line are the edges of $H_{X_{i^{\prime}}}$).
\[
M''=
\begin{blockarray}{c@{\hspace{5pt}}cccccc}
     & L(v_1) & L(v_2) & \cdots & L(v_{d+1}) & L(v_{i^{\prime}}) &  \\
\begin{block}{c(cccc|cc)}
  \phi_{i_1}=\phi_d & \phi_{i_1}(v_1) & \phi_{i_1}(v_2) & \cdots & \phi_{i_1}(v_{d+1}) & b_1 = d+2 & b_2 = a_{d+1} \\
  \phi_{i_2} & \phi_{i_2}(v_1) & \phi_{i_2}(v_2) & \cdots & \phi_{i_2}(v_{d+1}) & b_2 = a_{d+1} & b_3 \\
  \vdots & \vdots & \vdots &  & \vdots  & \vdots & \vdots \\
  \phi_{i_{\ell}} = \phi_{d+1} & \phi_{i_{\ell}}(v_1) & \phi_{i_{\ell}}(v_2) & \cdots & \phi_{i_{\ell}}(v_{d+1}) & b_{\ell} & b_{\ell+1}=d+1 \\
  \vdots & \vdots & \vdots &  & \vdots  & \vdots & \vdots \\
  \phi_{i_{d+1}} & \phi_{i_{d+1}}(v_1) & \phi_{i_{d+1}}(v_2) & \cdots & \phi_{i_{d+1}}(v_{d+1}) & b_{d+1} & b_{d+2} \\
\end{block}
\end{blockarray}
\]
Let $X:=(\{v_1,v_2,\ldots,v_{d+1}\} \setminus \{v_{i^{\prime}},v_{d+1}\}) \cup \{w_{a_{d+1}}\}$.
Then the subgraph of $H_X$ induced by edges $\{ \phi_{i_{j}}(X)^c: j=2,3, \ldots, \ell\} $ is a cycle. The  matrix $M'''$ illustrates the cycle $(b_2,b_3, \ldots, b_{\ell}, b_2)$. Note that the column indexed by $L(w_{a_{d+1}})$ in $M'''$ is identical to the column indexed by $L(v_{d+1})$ in $M''$, except that
$\phi_{i_{\ell}}(w_{a_{d+1}})=a_{d+2}$ and $\phi_{i_{\ell}}(v_{d+1})=a_{d+1}$. Thus $M'''$ is identical to $M''$, except that $\phi_{i_{\ell}}(w_{a_{d+1}})=a_{d+2}$ and $\phi_{i_{\ell}}(v_{d+1})=a_{d+1}$, and hence the graph 
$H_X$ is obtained from $H_{X_{i'}}$ by replacing the edge $b_lb_{l+1}$ with the edge $b_lb_2$.
In the matrix $M'''$ below, the pairs to the right of the vertical line are the edges of $H_X$.
\[
M'''=
\begin{blockarray}{c@{\hspace{5pt}}cccccc}
     & L(v_1) & L(v_2) & \cdots & L(w_{a_{d+1}}) & L(v_{i^{\prime}}) &  \\
\begin{block}{c(cccc|cc)}
  \phi_{i_1}=\phi_d & \phi_{i_1}(v_1) & \phi_{i_1}(v_2) & \cdots & \phi_{i_1}(w_{a_{d+1}}) & b_1 = d+2 & b_2 = a_{d+1} \\
  \phi_{i_2} & \phi_{i_2}(v_1) & \phi_{i_2}(v_2) & \cdots & \phi_{i_2}(w_{a_{d+1}}) & b_2 = a_{d+1} & b_3 \\
  \vdots & \vdots & \vdots &  & \vdots  & \vdots & \vdots \\
  \phi_{i_{\ell}} = \phi_{d+1} & \phi_{i_{\ell}}(v_1) & \phi_{i_{\ell}}(v_2) & \cdots &
  \phi_{i_{\ell}}(w_{a_{d+1}})=a_{d+2} & b_{\ell} & a_{d+1}=b_2 \\
  \vdots & \vdots & \vdots &  & \vdots  & \vdots & \vdots \\
  \phi_{i_{d+1}} & \phi_{i_{d+1}}(v_1) & \phi_{i_{d+1}}(v_2) & \cdots & \phi_{i_{d+1}}(w_{a_{d+1}}) & b_{d+1} & b_{d+2} \\
\end{block}
\end{blockarray}
\]
This completes the proof of Lemma \ref{bad_clique}.
\end{proof}

Let $G^{\prime}$ be the $d$-tree obtained from $G$ as follows:
For each $d$-clique $W$ of $G$, we add  $d+1$ new vertices $x^W_1,x^W_2,\ldots,x^W_{d+1}$
and 
add edges connecting each $x^W_i$ to all vertices in $W$. 
We extend the list assignment $L$ of $G$ to $G'$ by letting  $L(x^W_i) = [1,d+2] \setminus \{i\}$ for $1 \leq i\leq d+1$.
We show that $G^{\prime}$ does not have an $L$-list packing with size $d+1$.

Assume to the contrary that $\phi_1, \phi_2, \ldots, \phi_{d+1}$ is an $L$-list packing of $G'$.
By Claim~\ref{bad_clique}, there exist 
a $d$-clique $W$ of $G$ such that $H_W$ contains a cycle. Assume the cycle is induced by edge $\{\phi_i(W)^c: i \in I\}$ for some $\emptyset \ne I \subseteq [1,d+1]$. Then $|\bigcup_{i \in I} \phi_i(W)^c| = |I|$.

Assume $j \in \bigcup_{i \in I} \phi_i(W)^c \setminus \{d+2\}$. As $j \notin L(x^W_j)$, we have 
$| (\bigcup_{i \in I} \phi_i(W)^c) \cap L(x^W_j)| \le |I|-1$. 
On the other hand,  $\{\phi_i(x^W_j): i \in I\}$ is a subset of $ (\bigcup_{i \in I} \phi_i(W)^c) \cap L(x^W_j)$ of size $|I|$, a contradiction.

\section{Complexity}\label{section:complexity}

The proof of Theorem~\ref{thm:complexity} is almost parallel to the proof in \cite{FFLRSST11}, where it was proved that the \listcoloring problem, parameterized by the tree-width bound $t$, is fixed-parameter tractable and solvable in linear time for every fixed $t$.
 
If $k \geq 2t$, then the output of the \listpacking problem is always YES, and thus
we may assume that $k \leq 2t-1$.

We first translate a problem of list packing   of $G$ into a problem of list coloring a graph $ \widetilde{G}$.
For a graph $G$ and a $k$-assignment $L$ of $G$, let $\widetilde{G}$ be the Cartesian product of $G$ and the complete graph $K_k$, with $V(\widetilde{G})=\{x^v_i : v\in V(G), 1\leq i\leq k\}$ and $E(\widetilde{G})=\{x^u_ix^v_i : uv\in E(G), 1\leq i\leq k\}\cup \{x^v_ix^v_j : v\in V(G), 1\leq i< j\leq k\}$.
For each $v\in V(G)$, let $X_v:=\{x^v_i : 1\leq i\leq k\}\subseteq V(\widetilde{G})$.
Let $\widetilde{L}$ be a $k$-assignment of $\widetilde{G}$ such that $\widetilde{L}(x^v_i)=L(v)$ for any $v\in V(G)$ and any $i\in [1,k]$.
Then the following holds.

\begin{claim}\label{claim:cartesian product}
    $G$ admits an $L$-packing if and only if $\widetilde{G}$ is $\widetilde{L}$-colorable.
\end{claim}

\begin{proof}
    If $G$ admits an $L$-packing $\phi=(\phi_1,\phi_2,\ldots,\phi_k)$, then by setting $c(x^v_i)=\phi_i(v)$ for any $v\in V(G)$ and $i\in [1,k]$, we obtain an $\widetilde{L}$-coloring of $\widetilde{G}$.
    Conversely, let $c$ be an $\widetilde{L}$-coloring of $\widetilde{G}$.
    Then, by setting $\phi_i(v)=c(x^v_i)$ for any $v\in V(G)$ and $i\in [k]$, we obtain an $L$-packing $(\phi_1,\phi_2,\ldots ,\phi_k)$.
\end{proof}

Fellows et al.\cite{FFLRSST11} showed the following, which states that for the \listcoloring problem, we only have to consider list assignments using a small number of colors in total.

\begin{lemma}\label{lem:small list}
    If a graph $H$ of tree-width at most $r$ has a $k$-assignment $L_H$ such that $H$ is not $L_H$-colorable, then there exists another $k$-assignment $L'_H$ of $H$ such that $|\bigcup_{v\in V(H)}L'_H(v)|\leq (2r+1)k$ and $H$ is not $L'_H$-colorable.
\end{lemma}

Furthermore, by their proof of Lemma~\ref{lem:small list}, it is easy to see that if two vertices $u$ and $v$ of $H$ belong to the same bag of a fixed tree decomposition of width at most $r$ and have the same list in $L_H$, then they still have the same list in the new list assignment $L'_H$.
This leads to the following.

\begin{claim}
    Let $G$ be a graph of tree-width at most $t$.
    Then $\chi_l^{\star}(G)\leq k$ if and only if $\widetilde{G}$ is $\widetilde{L}$-colorable for any $k$-assignment of $\widetilde{L}$ of $\widetilde{G}$ satisfying the followings;
    \begin{enumerate}[label=(\alph*)]
        \item\label{list condition1} $\widetilde{L}(x^v_i)=\widetilde{L}(x^v_j)$ for any $v\in V(G)$ and $1\leq i<j\leq k$ and 
        \item\label{list condition2} $\bigl|\bigcup_{x^v_i\in V(\widetilde{G})}\widetilde{L}(x^v_i)\bigr|\leq (2k(t+1)-1)k$.
    \end{enumerate}
\end{claim}

\begin{proof}
    By Claim~\ref{claim:cartesian product}, for any $k$-assignment $\widetilde{L}$ of $\widetilde{G}$ satisfying \ref{list condition1}, there exists a $k$-assignment $L$ of $G$ such that the existence of an $L$-packing of $G$ is equivalent to the existence of an $\widetilde{L}$-coloring of $\widetilde{G}$.
    This implies the ``only if'' part.

    Conversely, suppose that $\chi_l^{\star}(G)>k$.
    Then there exists a $k$-assignment $L$ of $G$ such that $G$ does not admit an $L$-packing.
    By Claim~\ref{claim:cartesian product}, $\widetilde{G}$ is not $\widetilde{L}$ colorable for some $k$-assignment $\widetilde{L}$ satisfying \ref{list condition2}.
    As $\widetilde{G}$ has tree-width at most $k(t+1)-1$, we fix a tree-decomposition of $\widetilde{G}$ of width at most $k(t+1)-1$.
    Now we apply Lemma~\ref{lem:small list} to obtain a $k$-assignment $\widetilde{L}'$ of $\widetilde{G}$ such that $\bigl|\bigcup_{x^v_i\in V(\widetilde{G})}\widetilde{L}'(x^v_i)\bigr|\leq \bigl(2(k(t+1)-1)+1\bigr)k=(2k(t+1)-1)k$ and $\widetilde{G}$ is not $\widetilde{L}'$-colorable.
    For any vertex $v\in V(G)$, since $X_v$ is a clique in $\widetilde{G}$, $X_v$ is contained in a bag of the fixed tree-decomposition of $\widetilde{G}$.
    Thus, $\widetilde{L}'$ satisfies \ref{list condition1} as well, and this completes the ``if'' part.
\end{proof}

Let $C$ be a set of $(2k(t+1)-1)k$ colors.
Now we consider a graph $G^*$ obtained from $\widetilde{G}$ by joining $C$.
Then, any $k$-assignment $\widetilde{L}$ of $\widetilde{G}$ with \ref{list condition1} and \ref{list condition2} corresponds to a spanning subgraph $F_{\widetilde{L}}$ of $G^*$ such that
\begin{quote}
    $(*)$ $E(F_{\widetilde{L}})=E(\widetilde{G})\cup \bigcup_{v\in V(G)}\{x^v_ic : c\in C_v\}$, where $C_v\subseteq C$ is a set of $k$ colors assigned to $v$ by $L$,
\end{quote}
and vise versa.
For a subgraph $F_{\widetilde{L}}$ of $G^*$, an $\widetilde{L}$-coloring of $\widetilde{G}$ corresponds to a subgraph $M_{\widetilde{L}}$ of $F_{\widetilde{L}}$ such that
\begin{quote}
    $(**)$ every $x^v_i\in V(\widetilde{G})$ has exactly one neighbor in $C$ and $M_{\widetilde{L}}$ is triangle-free,
\end{quote}
and vise versa.
It is easy to see that each of $(*)$ and $(**)$ can be stated in MSO on $G^*$.
Since $G^*$ has tree-width at most $k(t+1)-1+|C|=k(t+1)-1+(2k(t+1)-1)k=O(t^3)$, by the Courcelle's Theorem, whether $\widetilde{G}$ has an $\widetilde{L}$-coloring for any $k$-assignment satisfying \ref{list condition1} and \ref{list condition2} can be solved in linear time, with fixed parameter $t$.
This completes the proof of Theorem~\ref{thm:complexity}.


\bibliographystyle{abbrv}
\bibliography{listpack}

\end{document}